\documentclass[11pt]{amsart}
\usepackage[utf8x]{inputenc}
\usepackage{amsthm,amssymb, amsmath}
\usepackage{graphicx}
\usepackage{stmaryrd}
\usepackage[all]{xy}

\newtheorem{theorem}{Theorem}[section]
\newtheorem{corollary}[theorem]{Corollary}

\newtheorem{lemma}[theorem]{Lemma}
\newtheorem{proposition}[theorem]{Proposition}

\newtheorem{definition}[theorem]{Definition}

\newcommand{\matr}[4]{
\left( \begin{array}{cc} #1 & #2 \\ #3 & #4 \end{array} \right)}
\newcommand{\vect}[2]{
\left( \begin{array}{c} #1 \\ #2 \end{array} \right)}

\title{Galois conjugates of entropies of real unimodal maps}
\author{Giulio Tiozzo}
\address{ICERM, Brown University, 
121 South Main St, Providence RI 02903.}
\email{Giulio\_Tiozzo@brown.edu}

\begin{document}
\maketitle
\section{Introduction}

A classical way to measure 
the complexity in the orbit structure of a dynamical system $f : X \to X$
is its  \emph{topological entropy} $h(f)$. 
When the system has a Markov partition, then its topological entropy 
is the logarithm of an algebraic number: 
in fact, if we call \emph{growth rate} of $f$ 
the quantity 
$$s(f) := e^{h(f)}$$
then $s(f)$ is the leading eigenvalue of the transition matrix associated to the 
partition. In this paper, we are interested in the relationship between the dynamical properties 
of $f$ and the algebraic properties of its growth rate.

By the Perron-Frobenius theorem it follows immediately that $s(f)$
must be 
a \emph{weak Perron number}, i.e. a real algebraic integer which is 
at least as large as the modulus of all its Galois conjugates. 
In \cite{Th}, Thurston asked the converse
 
\medskip 
\textbf{Question.} What algebraic integers arise as growth rates 
of dynamical systems with a Markov partition?
\medskip

The question makes sense in several contexts, e.g. for pseudo-Anosov maps of surfaces, 
as well as automorphisms of the free group. In this note we shall focus on \emph{multimodal maps}, 
i.e. continuous interval maps which have finitely many intervals of monotonicity (e.g., polynomial maps).
In this context, the condition of having a Markov partition can be reformulated by saying that 
a multimodal map is \emph{postcritically finite} if the orbits of all its critical 
points are finite. For these maps, the above question was settled in the following

\begin{theorem}[\cite{Th}]
The set of all growth rates of postcritically finite 
multimodal interval maps coincides with the set of all weak Perron numbers.
\end{theorem}

The question becomes more subtle when one restricts oneself to maps of a given degree. 
In particular, in the case of degree two, Thurston looked at the algebraic properties 
of growth rates of postcritically finite real quadratic polynomials; remarkably, he found out 
that the union of all their Galois conjugates exhibits a rich fractal structure
(Figure \ref{hspectrum}). Moreover, he claimed that such a fractal set is path-connected.
In this note we will formally introduce this object and study its geometry.

Let $f_c(z) := z^2+c$ be a real quadratic polynomial, with $c \in [-2, 1/4]$.
We shall call the map $f_c$ \emph{superattracting} if the critical point $z = 0$ is periodic.
Each superattracting parameter is the center of a hyperbolic component in the Mandelbrot set; 
let us denote by $M_0$ the set of all superattracting parameters. Moreover, 
if $\lambda$ is an algebraic number, we shall denote by $Gal(\lambda)$ the set 
of Galois conjugates of $\lambda$, i.e. the set of roots of its minimal polynomial. 


\begin{definition}
We shall call \emph{entropy spectrum} $\Sigma$ the closure of the set of Galois conjugates of 
growth rates of superattracting real quadratic polynomials:
$$\Sigma := \overline{ \bigcup_{c \in M_0 \cap \mathbb{R}} Gal(s(f_c)) }.$$
\end{definition}
The set $\Sigma$ is a compact subset of $\mathbb{C}$, and it displays a lot of structure 
(see Figures \ref{hspectrum}, \ref{third}). We will establish the following:

\begin{theorem} \label{main}
The set $\Sigma$ is path-connected and locally connected.
\end{theorem}

\begin{figure}[h!]
\includegraphics[width=0.95\textwidth]{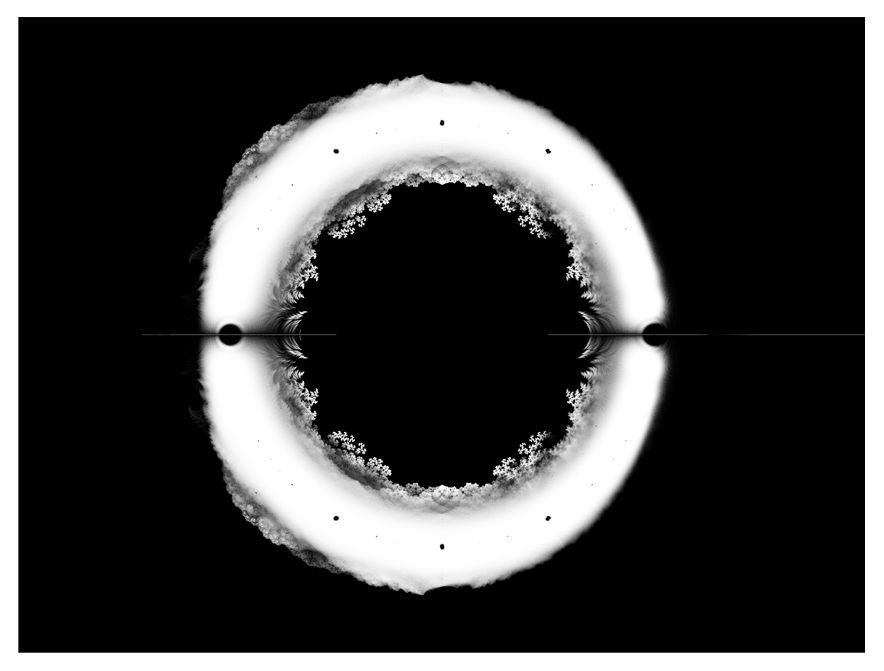}
\caption{The entropy spectrum $\Sigma$ for real quadratic polynomials. 
}
\label{hspectrum}
\end{figure}
The proof follows the techniques used by Bousch \cite{B1}, \cite{B2} and Odlyzko-Poonen \cite{OP}
to prove the path-connectivity of the sets of zeros of certain power series with prescribed coefficients
(see Figure \ref{pmu}).
The reason of this connection is the \emph{kneading theory} of Milnor and Thurston \cite{MT}. 
Indeed, they 
showed that for each map $f_c$ one can construct a certain power 
series $K_c(t)$, known as \emph{kneading determinant}, 
in such a way that the inverse of the growth rate of $f_c$ is a zero of $K_c(t)$.
Thus, the set $\Sigma$ is closely related to the set $\Sigma_{kn}$ of all 
zeros of all kneading determinants (see section \ref{section:kns}).

The proof of the main theorem will be split in several parts. 
We shall denote by $\mathbb{D} := \{ z \in \mathbb{C} \ : \ |z| < 1\}$ the unit disk in the complex plane, and 
by $\mathbb{E} := \{ z \in \mathbb{C} \ : \ |z| > 1\}$ the part of the plane 
outside the closed unit disk.
The part of $\Sigma$ inside the unit disk will be analyzed 
by comparing it to the set $\Sigma_{\pm 1}$ of zeros of all polynomials with coefficients $\pm 1$
(section \ref{section:pmu}); in fact, we shall prove that $\Sigma$ 
and $\Sigma_{\pm 1}$ coincide inside the unit disk (Proposition \ref{equalpmu}). 
On the other hand, the part outside the disk will require a different analysis. 
Namely, we shall first analyze the kneading set $\Sigma_{kn}$, proving that 
the intersection $\Sigma_{kn} \cap \mathbb{E}$ is connected and locally connected (section \ref{section:kns}). 
Finally, in order to prove the same statement about 
$\Sigma$, we shall then address the question of which polynomials given 
by the kneading theory are in fact irreducible. As we shall see (section \ref{section:irredrenorm}),  
this question is closely related to the combinatorics of renormalization. 
We shall also prove that $\Sigma$ and $\Sigma_{kn}$ both contain a neighbourhood 
of the unit circle (section \ref{section:neigh}), completing the proof. 

Finally, it is worth pointing out that fractal sets similar to $\Sigma$ can be constructed
using other families of quadratic polynomials. In particular, one can consider 
for each postcritically finite quadratic polynomial $f$ the restriction of $f$ to its Hubbard tree, 
and its growth rate will be an algebraic number.
Thus, one can construct for instance the set of Galois conjugates of growth rates of 
superattracting maps along any vein in the Mandelbrot set (see the Appendix for some pictures), 
or even consider all centers of all hyperbolic components at once.
The corresponding questions about the geometry of these sets are still open.




\subsection*{Acknowledgements} 
All the essential ideas go back to W. Thurston: this paper wants to be 
a step towards a more complete understanding of his last works, which are
extremely rich and deserve to be completed in detail.
I wish to thank C. T. McMullen for putting me in contact with Thurston's work, in particular 
the preprint \cite{Th}. I also thank S. Koch, Tan Lei and B. Poonen for useful conversations.

\section{Review of kneading theory}

Let $f(z) = z^2 + c$. We define the \emph{sign} $\epsilon(x)$ of a point $x \neq 0$ with respect to the partition 
given by the critical point to be  
$$\epsilon(x) := \left\{ \begin{array}{ll} -1 & \textup{ if }x < 0 \\
					+1 & \textup{ if }x > 0. 
\end{array} \right.$$
Moreover, for each $k \geq 1$ we define $\eta_k(x) := \epsilon(x) \epsilon(f(x)) \dots \epsilon(f^{k-1}(x))$.
If the forward orbit of $x$ does not contain the critical point, then the \emph{kneading series} of $x$ is defined as 
$$K(x, t) := 1 + \sum_{k = 1}^\infty \eta_k(x) t^k.$$
We now associate to each map $f_c$ a power series $K_c(t)$, known as \emph{kneading determinant}. 
If the critical point is not periodic, then we define
$$K_c(t) := K(c,t).$$
Otherwise, we define 
$$K_c(t) := \lim_{x \to c^+} K(x, t)$$
where the limit is taken over all subsequences such that $x$ does not map to the critical point.
The series $K_c(t)$ converges in the disk of unit radius, and its smallest real positive root is the inverse of the growth rate:
\begin{theorem}[\cite{MT}] \label{MTknead}
Let $s$ be the growth rate of $f_c$. Then the function $K_c(t)$ has no zeros 
on the interval $[0, 1/s)$, and if $s > 1$ we have 
$$K_c(1/s) = 0.$$
\end{theorem}

If the critical point is periodic of period $p$, then the coefficients of $K_c(t)$ are 
periodic, so the function $K_c(t)$ can be written in the form 
$$K_c(t) = \frac{P(t)}{1 - t^p}$$
where $P(t)$ is a polynomial of degree $p-1$ with coefficients in $\{ \pm 1\}$. We shall call 
$P(t)$ the \emph{kneading polynomial} of $f_c$. A power series is \emph{admissible} if it is the kneading 
determinant of  some real quadratic polynomial. Similarly, a polynomial $P(t)$ of degree $n$ is admissible if the
power series expansion of $P(t)/(1-t^{n+1})$ is admissible.
Admissible power series can be characterized in terms of the action of the shift operator on its coefficients.
In order to recall the criterion, let us say that a formal power series $\phi(t)$ is \emph{positive} if its first 
non-zero coefficient is positive, and that two formal power series satisfy $\phi_1(t) < \phi_2(t)$ if 
$\phi_2(t) - \phi_1(t)$ is positive.
Moreover, the absolute value $|\phi(t)|$ of a formal power series will equal $\phi(t)$ is $\phi(t) \geq 0$ 
and $-\phi(t)$ if $\phi(t) < 0$. The admissibility criterion is the following.

\begin{theorem}[\cite{MT}, Theorem 12.1] \label{admthm}
Let $$\phi(t) := 1+  \sum_{k = 1}^\infty  \epsilon_k t^k$$
be a power series, with $\epsilon_k \in \{ \pm 1 \}$.
The power series $\phi(t)$ is admissible if and only if 
$$\phi(t) \leq |\sum_{k = n}^\infty \epsilon_k t^k|$$ 
for each $n \geq 1$.
\end{theorem}

In particular, the theorem immediately implies the following sufficient condition, 
which we will use later.

\begin{corollary} \label{admissible}
Let $$\phi(t) := 1 + \sum_{k = 1}^\infty  \epsilon_k t^k$$
be a power series, with $\epsilon_k \in \{ \pm 1 \}$ and $\epsilon_1 = -1$.
Define the \emph{initial runlength} $N$ to be the number of consecutive equal 
symbols at the beginning of the sequence of coefficients:
$$N := \max \{ k \geq 1 : \epsilon_1 = \epsilon_{2} = \dots = \epsilon_{k} \}$$  
and the \emph{maximal runlength} $M$ to be the cardinality of the largest sequence 
of consecutive equal symbols, excluding the first one:
$$M := \max \{ k \geq 1 : \epsilon_n = \epsilon_{n+1} = \dots = \epsilon_{n+k-1} \textup{ for some }n \geq N+1\}.$$
If $N > M$, then the power series $\phi(t)$ is admissible. 
\end{corollary}

Moreover, let us recall that kneading determinants behave nicely under tuning operations. 
Indeed, let $f_{c_0}$ be a superattracting real polynomial of period $p$ and kneading determinant
$K_{c_0}(t) = P_{c_0}(t)/(1-t^p)$, and $f_{c_1}$ another real polynomial. 
Then their Douady-Hubbard tuning $f_{c_2} = f_{c_0} \star f_{c_1}$ has kneading polynomial (see \cite{Do1})
\begin{equation} \label{eq:renorm}
P_{c_2}(t) = P_{c_1}(t^p)P_{c_0}(t). 
\end{equation}

Let us conclude the section with a few basic observations on the geometry of $\Sigma$.
 
\begin{lemma} \label{basicincl}
We have the inclusion
$$\Sigma \subseteq \{ z \ : \ 1/2 \leq |z| \leq 2 \}.$$
\end{lemma}

\begin{proof}
Let $P(t) = \sum_{k = 0}^n \epsilon_k t^k$, with $|\epsilon_k| = 1$. Then 
if $|t| < 1/2$ we have 
$$|P(t)| \geq 1 - \sum_{k = 1}^n 2^{-n} > 0$$  
so there is no zero inside the disk of radius $1/2$. Taking the reciprocal polynomial 
proves that there is no root of modulus larger than $2$. The claim then follows 
because all kneading polynomials for superattracting maps have coefficients of unit modulus.
\end{proof}

Finally, one of the main results of Milnor and Thurston's kneading theory is the following 
monotonicity of entropy.

\begin{theorem}[\cite{MT}]
The growth rate $s(f_c)$ of $f_c(z) := z^2 + c$ is a continuous, decreasing function 
of the parameter $c \in [-2, 1/4]$.
\end{theorem}

Since $s(f_{1/4}) = 1$ and $s(f_{-2}) = 2$, by the density of hyperbolic components on the real 
line we get immediately the

\begin{corollary} \label{contains12}
The set $\Sigma$ contains the real interval $[1, 2]$. 
\end{corollary}



\section{Outside the unit disk : the kneading spectrum} \label{section:kns}

For the sake of exposition, let us first analyze a set which is related to $\Sigma$.
Let us define the \emph{kneading spectrum} $\Sigma_{kn}$ for real unimodal maps 
as the set of (inverses of) all zeros of kneading determinants: more precisely, we set

$$\Sigma_{kn} := \overline{\{ s \in \mathbb{C}^* \ : \ K_c(1/s) = 0 \textup{ for some superattracting }f_c\}}.$$
Since the growth rates are zeros of the reciprocals of kneading polynomials, then we have the inclusion 
$$\Sigma \subseteq \Sigma_{kn}.$$
However, it is not always true that kneading polynomials are irreducible (indeed, they are not inside 
small copies of the Mandelbrot set, see section \ref{section:irredrenorm}), so it is not obvious that the two sets are the same. 

In this section, we shall prove the following result.

\begin{proposition} \label{outside}
The set $\Sigma_{kn} \cap \{z : |z| \geq 1\}$ is connected, and the set 
$\Sigma_{kn} \cap \{z : |z| > 1\}$ is locally connected.
\end{proposition}

Let us first observe that the set $\Sigma_{kn}$ has the remarkable property of 
being closed under taking $n^{th}$-roots, and this is precisely because of renormalization. 

\begin{lemma} \label{nth-roots}
If $s \in \mathbb{C}^*$ belongs to $\Sigma_{kn}$ and $t^n = s$ for some $t \in \mathbb{C}^*$ and $n$ a positive integer, 
then $t$ also belongs to $\Sigma_{kn}$. As a consequence, $\Sigma_{kn}$ contains the unit circle $S^1$. 
\end{lemma}

\begin{proof}
Let $s$ be such that $K_c(1/s) =0$ with $f_c(z)$ a superattracting real quadratic polynomial, and let $t \in \mathbb{C}$ 
such that $t^n = s$. Now let us pick $f_{c_1}$ a superattracting real quadratic polynomial with critical orbit of period $n$, 
and construct the tuned map $f_{c_2} := f_{c_1} \star f_{c}$. Then by equation \eqref{eq:renorm} we have that 
$K_{c_2}(z) = K_{c}(z^n)P_{c_1}(z)$, and by evaluating it for $z = t^{-1}$ we get $K_{c_2}(t^{-1}) =  K_{c}(t^{-n})P_{c_1}(t^{-1}) = 0$
since $t^{-n} = s^{-1}$, hence $t$ also belongs to $\Sigma_{kn}$. The claim then follows by taking closures. 
\end{proof}

Let us first observe that, since periodic kneading sequences are dense in the set of 
admissible kneading sequences, we can drop the closure if we admit kneading determinants 
of all real maps: that is, we have the identity
$$\Sigma_{kn} \cap \mathbb{E} = \{ s \in \mathbb{E} \ : \ K_c(1/s) = 0 \textup{ for some }c \in [-2, 1/4]\}.$$

The fundamental idea then is that we can associate to each parameter $c$ a discrete subset of the disk, 
namely the set of zeros of $K_c(t)$, in a continuous way, and we are interested in studying the union of all such sets. 
It is thus natural to consider the three-dimensional set 
\begin{equation} \label{sigmahat}
\widehat{\Sigma} := \{ (c, z) \in [-2, 1/4] \times \mathbb{D} \ s.t. \ K_c(z) = 0 \}
\end{equation}
which ``fibers'' over $[-2, 1/4]$ by taking the projection $\pi_1$ onto the first coordinate, and 
each fiber of $\pi_1$ is the set of zeros of $K_c(t)$. 


We will actually prove that $\widehat{\Sigma}$ is connected and locally connected: 
the Proposition then follows since the set $\Sigma_{kn} \cap \mathbb{E}$ is just obtained by taking 
the projection of $\widehat{\Sigma}$ onto the second coordinate, and then inverting through the unit circle via the map $z \mapsto 1/z$.

Let $Com \ V$ denote the space of compact subsets of a compact metric space $V$, with the Hausdorff topology.
Moreover, let $\widetilde{\mathbb{D}} := \mathbb{D} \cup \{ \infty \}$ be the one-point compactification of the 
unit disk. 
If $f$ is a holomorphic function in the unit disk, the \emph{trace} of $f$ is defined as the set of zeros 
of $f$: 
$$\textup{tr }f := \{ z \in \mathbb{D} \ : \ f(z) = 0 \} \cup \{ \infty\}.$$
By Rouch\'e's theorem, the map $\textup{tr }f : \mathcal{O}(\mathbb{D}) \to Com \ \widetilde{\mathbb{D}}$ is continuous
at $f$ as long as $f$ is not identicallly $0$. 
Let us now verify continuity for kneading determinants:
\begin{proposition}
The map 
$\textup{Tr}: [-2, 1/4] \to Com \ \widetilde{\mathbb{D}}$ given by 
$$\textup{Tr}(c) := \{ z \in \mathbb{D} \ : \ K_c(z) = 0 \} \cup \{\infty \}$$
is continuous in the Hausdorff topology.
\end{proposition}

\begin{proof}
Let us consider the map $\Phi : [-2, 1/4] \to \mathcal{O}(\mathbb{D})$ given by $\Phi(c) := K_c(t)$. 
If the critical point is not periodic for $f_c$, then $\Phi$ is continuous at $c$ because it is continuous in 
the topology of formal power series. 
Otherwise, if the critical point has period $p$, we have 
$$\lim_{s \to c^{+}} \Phi(s) = \frac{P(t)}{1-t^p} \qquad \lim_{s \to c^{-}} \Phi(s) = \frac{P(t)}{1+t^p}$$
where $P(t)$ is a polynomial of degree $p-1$; thus the two limit functions have the same zero sets inside the unit disk,
so the map is still continuous.
\end{proof}

\begin{lemma}[\cite{B2}] \label{lclemma}
Let $\Lambda$ a topological space and $V$ a compact metric space. Let $t : \Lambda \to Com \ V$ be a continuous map, 
and denote as $t(\Lambda)$ the union 
$$t(\Lambda) := \bigcup_{\lambda \in \Lambda} t(\lambda).$$
Then the following are true:
\begin{enumerate}
 \item suppose $\Lambda$ is connected and there exists $\lambda_0 \in \Lambda$ such that 
$t(\lambda_0)$ is connected; then $t(\Lambda)$ is connected.
\item Suppose $\Lambda$ is compact and locally connected, and let $U \subset V$ be an open subset 
such that $t(\lambda) \cap U$ is discrete for each $\lambda \in \Lambda$.
Then $t(\Lambda) \cap U$ is locally connected.
\end{enumerate}

\end{lemma}

\begin{proof}[Proof of Proposition \ref{outside}]
We apply the Lemma to $\Lambda = [-2, 1/4]$ (which is obviously connected and locally connected), $V = \widetilde{\mathbb{D}}$, $U = \mathbb{D}$.
Since $K_0(t)$ has no zeros inside the unit disk, then $\textup{Tr}(0) = \{\infty\}$ is connected, so 
by (1) the one-point compactification of $\Sigma_{kn} \cap \mathbb{E}$ is connected. Since by Lemma \ref{nth-roots} $\Sigma_{kn}$ contains $S^1$ which is 
connected, then $(\Sigma_{kn} \cap \mathbb{E}) \cup S^1$ is also connected.
Since no kneading determinant is identically zero, for each $c$ the set of zeros of $K_c(t)$ inside the unit disk is discrete, 
so by (2) we get that that $\Sigma_{kn} \cap \mathbb{E}$ is locally connected.
\end{proof}

A (direct) proof of the path-connectivity of $\Sigma_{kn}$ will be given in section \ref{section:lifting}.
It is worth mentioning that a three-dimensional object very close to $\widehat{\Sigma}$ appears in Thurston's 
paper \cite{Th}, where it is called the ``master teapot''.

\section{Irreducible polynomials}

In order to study the Galois conjugates of growth rates we need to find
their minimal polynomials. In particular, since we know that 
kneading polynomials vanish on the growth rate, they coincide 
with the minimal polynomials once we prove they are irreducible.
To construct irreducible polynomials we shall use the next two algebraic lemmas.
The following observation is due to B. Poonen.

\begin{lemma} \label{irred}
Let $d = 2^n -1$ with $n \geq 1$, and choose a sequence $\epsilon_0, \epsilon_1, \dots, \epsilon_n$ with each 
$\epsilon_k \in \{\pm 1\}$, such that 
$\sum_{k = 0}^d \epsilon_k \equiv 2 \ \mod 4$. Then the polynomial 
$$f(x) := \epsilon_0 + \epsilon_1 x + \dots + \epsilon_{d} x^{d}$$
is irreducible in $\mathbb{Z}[x]$.
\end{lemma}

\begin{proof}
We apply Eisenstein's criterion to $g(x) := f(x + 1)$. Indeed, reducing modulo $2$, 
$$xf(x+1) \equiv \sum_{k=0}^d x(x+1)^k \equiv  (x+1)^{d+1} - 1 \equiv x^{d+1}$$
where in the last equation we used that $d+1$ is a power of $2$. Thus, we have $g(x) \equiv x^d$ modulo $2$, 
while $g(0) = \sum_{k = 0}^d \epsilon_k$ is divisible by $2$ but not by $4$ by hypothesis and Eisenstein's criterion 
can be applied.
\end{proof}

\begin{lemma} \label{irredhigh}
Let $f(x) := 1 + \sum_{k = 1}^d \epsilon_k x^k$ be a polynomial, with $\epsilon_k \in \{ \pm 1\}$ for all $1 \leq k \leq d$ and 
$\epsilon_k = -1$ for some $k$. If $f(x)$ is irreducible in $\mathbb{Z}[x]$, then for all $n \geq 1$ the polynomial $f(x^{2^n})$
is irreducible in $\mathbb{Z}[x]$.
\end{lemma}

\begin{proof}
Suppose by contradiction $n \geq 1$ is the minimal integer for which $f(x^{2^n})$ is not irreducible. 
Thus, there exists a (unique) factorization 
$$f(x^{2^n}) = g_1(x) \dots g_r(x)$$ 
where $g_1(x), \dots, g_r(x)$ are irreducible polynomials with $g_i(0) = 1$ for each $i$.
By substituting $x$ with $-x$, we get
$$f(x^{2^n}) = g_1(-x) \dots g_r(-x)$$ 
so by uniqueness of the factorization there exists an involution 
$\sigma : \{1, \dots, r \} \to \{1, \dots, r \}$ such that for each $i$ we have 
$g_i(-x) = g_{\sigma(i)}(x)$. If the involution $\sigma$ has a fixed point $i$, then  
$g_i(x)$ is of the form $g_i(x) = h_1(x^2)$ for some $h_1(x) \in \mathbb{Z}[x]$, which implies that 
$f(x^{2^n})$ can be factored as 
$$f(x^{2^n}) = h_1(x^2) h_2(x^2) \qquad h_1(x), h_2(x) \in \mathbb{Z}[x]$$
so $f(x^{2^{n-1}})$ is also reducible, contradicting the minimality of $n$. Hence the involution $\sigma$ 
has no fixed point and, by grouping together the factors $g_k(x)$, we have the factorization 
\begin{equation} \label{factpoly}
f(x^{2^n}) = g(x) g(-x) 
\end{equation}
for some $g(x) \in \mathbb{Z}[x]$. We shall now see that this is impossible, by comparing coefficients 
on both sides of equation \eqref{factpoly}. 
Let us denote the coefficients by  
$f(x^{2^n}) := \sum_{k = 0}^{2^n d} b_{k} x^{k}$ and $g(x) := \sum_{k=0}^{2^{n-1}d} a_k x^k$.
Let us look at equation \eqref{factpoly}. Since $a_0^2 = b_0 = 1$, then $a_0 = \pm 1$.
For each $1 \leq k \leq 2^{n-1}d$, the coefficient of $x^{2k}$ is of type
\begin{equation} \label{comparecoeff}
b_{2k} = \sum_{j = j_0}^{k-1} \pm 2 a_j a_{2k-j} \pm a_{k}^2 
\end{equation}
where $j_0 := \max\{0, 2k - 2^{n-1}d\}$. As a consequence, we see that $a_k$ is even if and only if $b_{2k}$ is even. Thus we get the following congruences:
\begin{equation} \label{congruences}
\left\{\begin{array}{ll}
a_j \equiv 0 \mod 2 & \textup{ for } 1  \leq j \leq 2^{n-2} \textup{ (if }n \geq 2\textup{)}\\
a_{j 2^{n-1}} \equiv 1 \mod 2 & \textup{ for } 1\leq j \leq d.
\end{array}\right.
\end{equation}
We now have two cases:
\begin{itemize}
 \item Suppose $n \geq 2$. Then if we look at equation \eqref{comparecoeff} for $k = 2^{n-2}$, 
we get the equality 
$$ 2 a_0 a_{2^{n-1}} + \sum_{j = 1}^{2^{n-2}-1} \pm 2 a_j a_{2^{n-1}-j} \pm a_{2^{n-2}}^2 = 0. $$
By equation \eqref{congruences}, every term in the sum except possibly for the first one is multiple of $4$, so the first term 
$2 a_0 a_{2^{n-1}}  = \pm 2 a_{2^{n-1}}$ must also be multiple of $4$, so $a_{2^{n-1}}$ is even. However, this contradicts
the second line of equation \eqref{congruences}.
\item Suppose $n = 1$. Then from equation \eqref{congruences} we have for each $1 \leq k \leq d$ that $a_k$ is odd, and 
equation \eqref{comparecoeff} becomes of the form  
$$b_{2k} = 2 N_k + (-1)^k a_k^2$$
where $N_k$ has the same parity as the number of terms under the summation symbol in  \eqref{comparecoeff}, which is 
$\min \{k, d-k \}$. Now, by analyzing the previous equation modulo $4$ we realize that 
$b_{2k}$ cannot be $-1$ for any $1\leq k \leq d/2$, so we must have $b_{2k} = 1$ for all $1\leq k \leq d/2$.
Moreover, for $k > d/2$ either $d$ is even and $b_{2k} = 1$ for all $d/2 \leq k \leq d$,   
or $d$ is odd and $b_{2k} = -1$ for all $d/2 \leq k \leq d$. In the first case we contradict the initial hypothesis on $f(x)$
since all its coefficients equal $+1$; in the second case, we also get a contradiction 
because we obtain that $f(x) = (\sum_{k = 0}^{\frac{d-1}{2}} x^k) (1 - x^{\frac{d+1}{2}})$ is not irreducible.
\end{itemize}
\end{proof}

Note that the condition on the coefficients of $f(x)$ not being all equal is essential: indeed, 
the polynomial $f(x) := \sum_{k = 0}^{p-1} x^k$ can be irreducible (e.g. for $p = 5$), 
but $f(x^2)$ never is. Related issues on irreducibility and its relationship 
to renormalization will be discussed in sections \ref{section:irredrenorm} and \ref{s:remirr}.

\section{Inside the disk: roots of polynomials with coefficients $\pm 1$} \label{section:pmu}

Since the series $K_c(t)$ need not converge outside the unit disk, then the set 
of zeros of $K_c(t)$ outside the disk need not (and probably does not) vary continuously
as a function of $c$. However, it turns out that the set $\Sigma \cap \mathbb{D}$ coincides
with another set which has a natural parameterization by a path-connected set.
Let $\Sigma_{\pm 1}$ be the set of zeros of power series with coefficients $+1$ or $-1$:
$$\Sigma_{\pm 1} := \left\{ s \in \mathbb{D} \ : \  \sum_{k = 1}^\infty \epsilon_k s^k = 0 
\textup{ for some } \epsilon_k \in \{-1, +1\}^\mathbb{N} \right\}.$$
The set $\Sigma_{\pm 1}$ 
was considered by Bousch \cite{B1}, \cite{B2} in connection with the dynamics of certain 
iterated function systems (IFS). In fact, for each $s \in \mathbb{D}$, 
the IFS given by $z \mapsto sz \pm 1$ has a compact attractor $K_s$, and a parameter $s \in \mathbb{D}$ belongs to 
$\Sigma_{\pm 1}$ if and only if $K_s$ contains the ``critical point'' $z = 0$.

The set $\Sigma_{\pm 1}$ is naturally parameterized by the Cantor set $\{ \pm 1\}^\mathbb{N}$; 
by producing a path-connected quotient of the Cantor set which parameterizes 
$\Sigma_{\pm 1}$, Bousch proved the following

\begin{theorem} \label{pmulc}
The set $\Sigma_{\pm 1}$ is path-connected and locally connected. 
\end{theorem}

We shall now see that the intersections of the two sets with the unit disk are the same:

\begin{figure}
\includegraphics[width=0.8\textwidth]{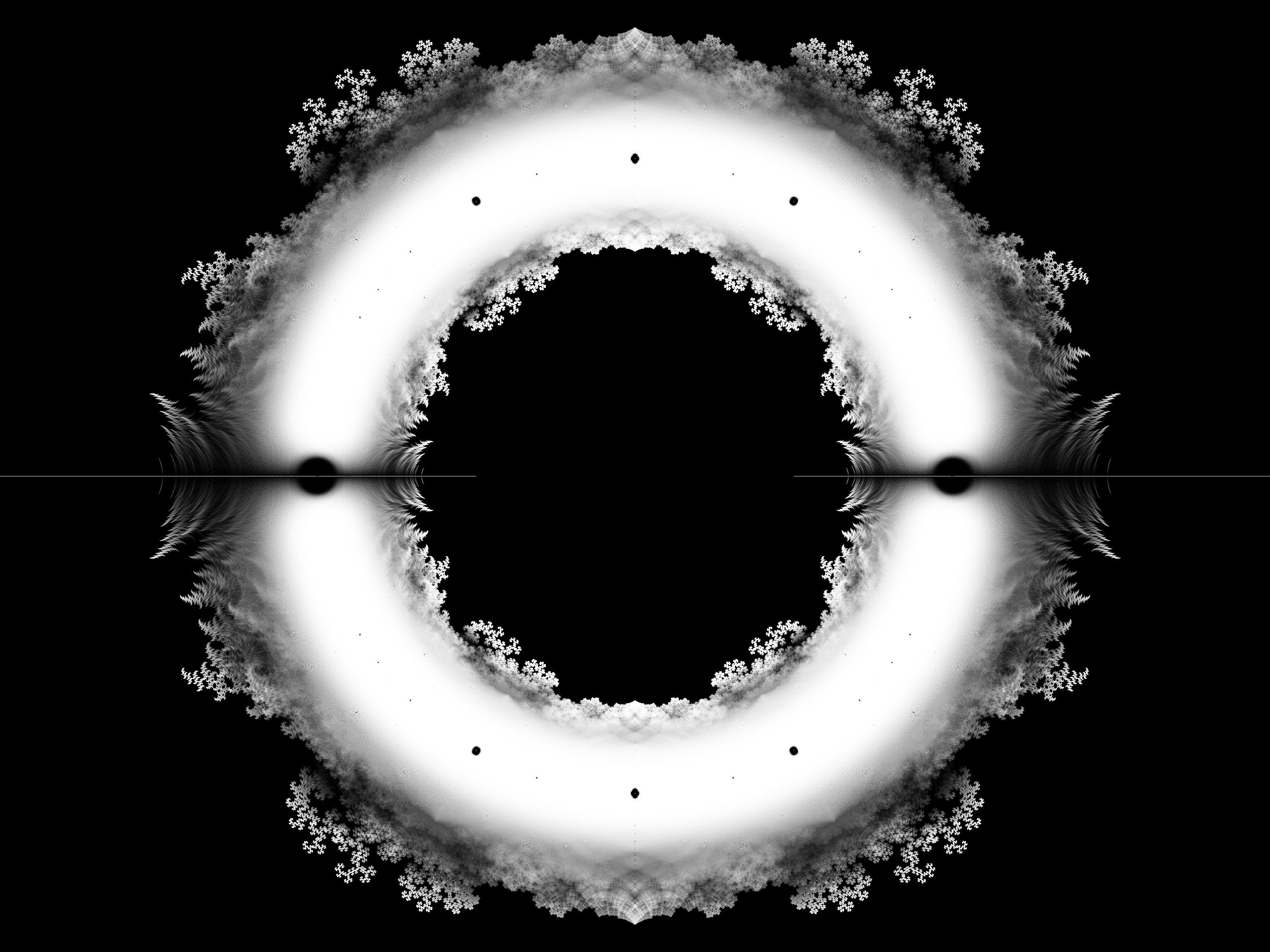}
\caption{The set of zeros of all polynomials with coefficients $\pm 1$.}
\label{pmu}
\end{figure}

\begin{proposition} \label{equalpmu}
We have the equality 
$$\Sigma \cap \mathbb{D} = \Sigma_{\pm 1} \cap \mathbb{D}.$$  
\end{proposition}

Then, using Theorem \ref{pmulc}, we get the

\begin{corollary} \label{inside}
The set $\Sigma \cap \mathbb{D}$ is path-connected
and locally connected.
\end{corollary}

\begin{proof}[Proof of Proposition \ref{equalpmu}]
Since the kneading determinants have coefficients $\pm 1$, it is clear that 
$\Sigma \cap \mathbb{D} \subseteq \Sigma_{\pm 1} \cap \mathbb{D}$.
In order to prove the other inclusion, let
$$\phi(t) := \sum_{k = 0}^\infty \epsilon_k t^k$$
be any power series with $\epsilon_i \in \{ \pm 1 \}$, and fix $n \geq 1$.
Let $N_0$ be the maximum number of consecutive equal digits in the sequence $(\epsilon_0, \dots, \epsilon_n)$:
$$N_0 := \max\{ k \ : \  \epsilon_i = \epsilon_{i+1} = \dots = \epsilon_{i+k-1} \textup{ for some } 0 \leq i \leq n-k+1 \}.$$
Then for each $N > N_0 + 1$ and each choice of $\eta \in \{\pm 1\}$, the polynomial
$$P_n(t) := 1 - \sum_{k = 1}^{N} t^k + \eta t^{N+1} + \sum_{k = 0}^{n} \epsilon_{n-k} t^{N+2+k}$$
is admissible. 
Moreover, by construction the first $n$  coefficients of its 
reciprocal polynomial $Q_n(t) := t^{N+n+2}P_n(t^{-1})$ coincide with the first $n$ coefficients 
of $\phi(t)$; thus by Rouch\'e's theorem each zero of $\phi(t)$ inside the unit disk is approximated by a sequence of zeros of $Q_n(t)$.
In addition, for each $n$ we can pick $N > N_0 +1$ and such that $n+N+3$ is a power of $2$, and we are free to choose
$\eta \in \{\pm 1\}$ such that $1 - N - \eta + \sum_{k = 0}^n \epsilon_k \equiv 2 \mod 4$. 
This way, by Lemma \ref{irred} the polynomials $Q_n(t)$ are irreducible, so the zeros of $Q_n(t)$ belong 
to $\Sigma$ and the claim is proven.
\end{proof}

The essential idea in the previous proof is that every sequence arises as suffix of an admissible 
kneading sequence: note that we cannot prove such an identity for the part of $\Sigma$ outside the disk because 
not every sequence arises as prefix of an admissible sequence, and indeed the pictures suggest that $\Sigma$ is smaller 
than $\Sigma_{\pm 1}$.



\section{A neighbourhood of the circle} \label{section:neigh}

Let us now prove that the set $\Sigma$ (hence also $\Sigma_{kn}$) contains a neighbourhood of the unit circle, 
as can be seen from Figure \ref{third}.

\begin{proposition} \label{annulus}
There exists $R > 1$ such that the inclusion
$$\{z :  R^{-1} < |z| < R \} \subseteq \Sigma$$
holds.
\end{proposition}

\begin{figure} 
\includegraphics[width=0.8\textwidth]{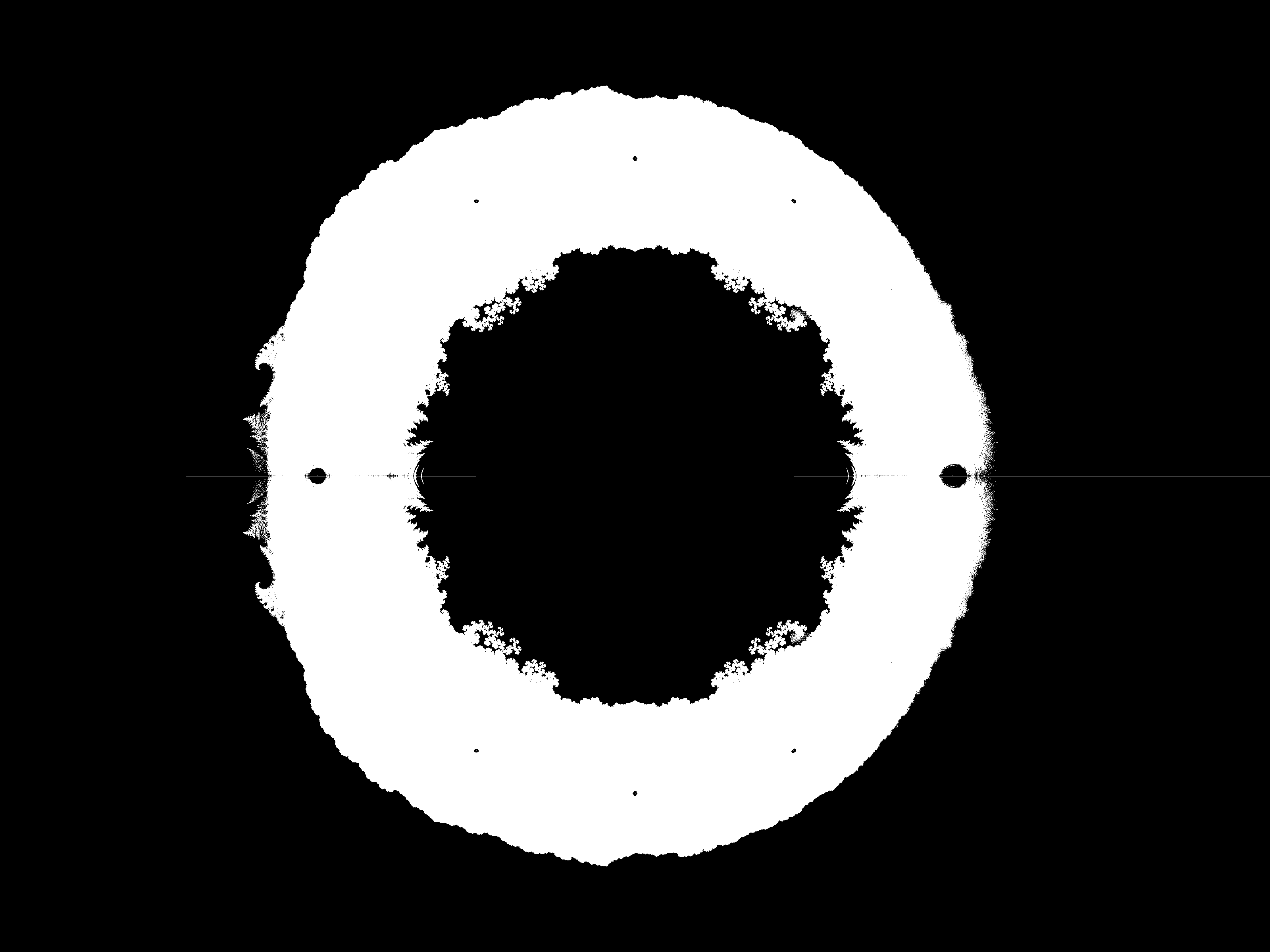} 
\caption{The boundary of the set $\Sigma$.}
\label{third}
\end{figure}

Bousch (\cite{B1}, Proposition 2) proves that the set $\Sigma_{\pm 1}$ contains the annulus $\{ z : 2^{-1/4} < |z| < 1 \}$, so by Proposition \ref{equalpmu} 
it is enough to prove that $\Sigma$ contains an annulus outside the unit circle, i.e. a set of the form $\{z : 1 < |z| < R \}$ for some $R > 1$. 
We shall use the following lemma (in the spirit of \cite{OP}, Lemma 3.1):

\begin{lemma} \label{boundarycrit}
Let $z \in \mathbb{D}$, and $m \geq 3$ an integer. Denote by $\Lambda$ the finite set
$$\Lambda := \{ \epsilon_0 + \epsilon_1 z + \dots + \epsilon_{m-1} z^{m-1} \ : \ \epsilon_k \in \{ \pm 1 \} \textup{, not all }\epsilon_k \textup{ equal}\}$$
and suppose there exists a bounded subset $B$ of $\mathbb{C}$ and an integer $n \geq 1$ such that the following hold: 
\begin{enumerate}
 \item 
we have the inclusion
$$B \subseteq \bigcup_{(w_1, \dots, w_n) \in \Lambda^n} w_1 + w_2 z^m + \dots + w_n z^{m(n-1)} + z^{mn} B;$$
\item 
$B$ contains the point 
$u_{2m}(z) := \frac{1 - z - z^2 - \dots - z^{2m-1}}{z^{2m}}.$ 
\end{enumerate}
Then $z^{-1}$ belongs to $\Sigma$.
\end{lemma}

\begin{proof}
From (2) and (1), we can write 
$$u = w_1 + w_2 z^m + \dots + w_n z^{m(n-1)} + z^{mn} b_1$$
with $u = u_{2m}(z)$ and some $(w_1, \dots, w_n) \in \Lambda^n$ and $b_1 \in B$; now, applying (1) recursively to $b_1$ we can find 
a sequence $\{b_N\}_{N \geq 1}$ of elements of $B$ and a sequence $\{w_k\}_{k \geq 1}$ of elements of $\Lambda$ such that for each $N$ we can write
$$u = \sum_{k = 1}^{n N} w_k z^{m(k-1)} + z^{m n N} b_N; $$
now, since $|z| < 1$ and $B$ is bounded we have in the limit
$$u = \sum_{k = 1}^\infty w_k z^{mk}$$
which can be rewritten as
$$0 = 1 - z - z^2 - \dots - z^{2m-1} + \sum_{k = 2m}^\infty \eta_k z^{k} \qquad \textup{with }\eta_k \in \{ \pm 1 \}.$$
Since we initially chose the $\epsilon_k$ not to be all equal, then the sequence $(\eta_k)_{k \geq 2m}$ does not 
contain any subsequence of $2m -1$ consecutive equal symbols, so the above power series is admissible and $z^{-1}$ belongs to the kneading 
spectrum $\Sigma_{kn}$.

In order to prove that $z^{-1}$ belongs to $\Sigma$, we still need to check that 
we can construct a sequence of admissible, irreducible polynomials whose coefficients converge to the 
sequence $(\eta_k)_{k \geq 0}$.
For each $N$, let us consider the truncation $(\eta_k)_{0 \leq k \leq 2^N-1}$ of the sequence $\eta$: 
if the sum $S := \sum_{k = 0}^{2^N-1} \eta_k$ is congruent to $2$ modulo $4$, then by Lemma \ref{irred} 
the polynomial $P_N(t) := \sum_{k = 0}^{2^N-1} \eta_k t^k$ is irreducible and admissible.
If the sum $S$ of the coefficients is instead divisible by $4$, we can flip one of the symbols $\eta_k$ so 
that the sum becomes congruent to $2$ and the sequence remains admissible. Precisely, 
we can find an index $k_0$ with $\max\{2m, 2^{N-1}\} \leq k_0 \leq 2^N-1$ such that the sequence
$(\eta'_k)_{k \geq 2m}$ defined as 
$$\eta'_k = \left\{ \begin{array}{ll} \eta_k & \textup{if }k \neq k_0 \\
                     -\eta_k & \textup{if }k = k_0
                    \end{array} \right.$$
still has at most $2m-2$ consecutive equal symbols\footnote{In general, the following 
is true: if $\sigma \in \{\pm 1 \}^n$ is any finite sequence and its maximum number of 
consecutive equal symbols is $M$, then we can flip one digit of $\sigma$ such that the new sequence $\sigma'$ has at most 
$\max\{3, M\}$ consecutive equal symbols.}, so that now
$\sum \eta'_k \equiv 2 \mod 4$ and the polynomial $P_N(t) := \sum_{k = 0}^{2^N-1} \eta'_k t^k$
is irreducible and admissible.
\end{proof}

We shall apply the lemma by taking the set $B$ to be a large ball around the origin: we shall 
need the following elementary lemma about convex sets, whose proof we postpone to the appendix.

\begin{lemma} \label{paracrit}
Let $v_1, \dots, v_n \in \mathbb{R}^d$ be non-zero vectors which span $\mathbb{R}^d$, and suppose 
that their convex hull $\Lambda$ contains the origin in its interior. Then there exists $R > 0$
such that any ball $B$ of radius at least $R$ centered at the origin satisfies the inclusion
$$\overline{B} \subseteq \textup{ int }\bigcup_{i = 1}^n (v_i + B)$$
where $\overline{B}$ denotes the closure and $\textup{int}$ the interior part.
\end{lemma}

\begin{proposition}
Let $z \in S^1$, $z \neq \pm 1$. Then a neighbourhood of $z$ is contained in $\Sigma$. 
\end{proposition}

\begin{proof}
Given $\xi \in S^1$, $\xi \neq \pm 1$, let us choose an integer $m \geq 3$  and coefficients $\epsilon_0, \epsilon_2, \dots, \epsilon_{m-1} \in \{ \pm 1 \}$ 
such that $\xi^m \neq \pm 1$, the vector $v := \sum_{k = 0}^{m-1} \epsilon_k \xi^k$ is non-zero and 
the $\epsilon_k$ are not all equal\footnote{This can always be done: e.g., if $\xi$ is not a $6^{th}$ root of unity we can choose $m = 3$ and $v = 1 - \xi + \xi^2$; 
if $\xi^3 = 1$, pick $m = 4$ and $v = 1 + \xi + \xi^2 - \xi^3$; finally, if $\xi^3  = -1$, pick $m = 4$ and $v = 1 - \xi + \xi^2 + \xi^3$.}.
Now, the four points in the set  
$$\Lambda := \{ \pm v \pm \xi^m v \}$$
are the vertices of a parallelogram which contains the origin in its interior, hence by Lemma \ref{paracrit} there exists a ball $B$ 
centered at the origin such that the inclusion
\begin{equation} \label{eq:incl}
B \subseteq \textup{ int }\bigcup_{w \in \Lambda} \left( w + \xi^{2m} B \right)
\end{equation}
holds. Moreover, we can choose the radius of $B$ to be large enough so that the point 
$u := \frac{1 - \xi - \xi^2 - \dots - \xi^{2m-1}}{\xi^{2m}}$ belongs to the interior of $B$.
Now, we see that the conditions of Lemma \ref{boundarycrit} are verified for each $z \in \mathbb{D}$
in a neighbourhood of $\xi$, so by the Lemma the set $\Sigma$ contains a neighbourhood of $\xi^{-1}$
and since this holds for all $\xi \in S^1 \setminus \{ \pm 1 \}$ the claim is proven.  
\end{proof}

\begin{proposition} \label{onecontained}
The points $z = \pm 1$ belong to the interior of the set $\Sigma$.
\end{proposition}

The proof in this case is a bit more complicated, so it will be postponed to the appendix.
It is still based on Lemma \ref{boundarycrit}, but we can  no longer choose a large ball 
to play the role of the bounded set $B$: instead, 
as in the proof of (\cite{OP}, Proposition 3.3), we have to choose a parallelogram whose shape
 varies with $z$.

\section{$\Sigma$ is path-connected}

Let us finally turn to the proof of the following 

\begin{theorem} \label{sigmapc}
The set $\Sigma \cap \{ z \ : \ |z| \geq 1 \}$ is path-connected.
\end{theorem}

\subsection{Lifting lemma and path-connectivity} \label{section:lifting}

The essential idea to prove path-connectivity is that the three-dimensional 
set $\widehat{\Sigma}$ defined in equation \eqref{sigmahat} ``fibers'' over an interval which is path-connected, so 
we can lift continuous paths in the base to continuous paths in $\widehat{\Sigma}$, 
and then project them to the other coordinate to get a continuous path in 
$\Sigma_{kn}$ or $\Sigma$. However, the issue of irreducibility of kneading 
determinants creates further complications.
 
The following topological tool is proven in \cite{OP}, where it is attributed
to D. des Jardins and E. Knill.

\begin{lemma}[\cite{OP}, Lemma 5.1] \label{lifting}
Let $M$ be a Hausdorff topological space and let $\pi : M^n \to M^n/S_n$ be the projection map onto the 
set of unordered $n$-tuples.
Then every continuous map $f : [0, 1] \to M^n/S_n$ can be lifted to 
a continuous map $g : [0, 1] \to M^n$ such that $f = \pi \circ g$.
\end{lemma}

Given $R < 1$, we shall denote by $\mathbb{D}_R$ the disk centered at the origin of radius $R$, and by $\widetilde{\mathbb{D}}_R$ 
the one-point compactification of $\mathbb{D}_R$. 
Moreover, we shall use the following set:
$$\widehat{\Sigma}_{I, R} := \{ (c, z) \in I \times \mathbb{D}_R \ : \ K_c(z) = 0 \} \cup I \times \{\infty\}$$
where $\infty$ denotes the boundary point of the one-point compactication of $\mathbb{D}_R$.

By applying the previous lemma to kneading determinants, we get the following path lifting 
property.

\begin{lemma} \label{liftknead}
Let $I := [a, b] \subseteq [-2, 1/4]$ a closed interval, $R < 1$ and $z \in \mathbb{D}_R$ such that 
$K_a(z) = 0$.
Then there exists a continuous path $\gamma : [a, b] \rightarrow \widetilde{\mathbb{D}}_R$ such that 
$\gamma(a) = z$ and for each $x \in [a, b]$ we have that 
$(x, \gamma(x))$ belongs to $\widehat{\Sigma}_{I, R}$.
\end{lemma}

\begin{proof} 
Since the coefficients of kneading determinants are universally bounded, then there exists 
(e.g. by Jensen's theorem, see \cite{OP}, Proposition 2.1)
a constant $N$, depending only on $R$, such that any kneading determinant $K_c(z)$ has at most $N$ roots, 
counted with multiplicities, inside the disk of radius $R$.
Thus we can 
define the map 
$$\Phi : I \to \underbrace{(\widetilde{\mathbb{D}}_R \times \dots \times \widetilde{\mathbb{D}}_R)}_{N \textup{ times}}/S_N$$
by taking $\Phi(c)$ to be the roots of $K_c(z)$ which lie inside $\mathbb{D}_R$, counted with multiplicities; if there are fewer 
than $N$ roots, then we take the remaining points to be $\infty$. The map $\Phi$ is continuous by Rouch\'e's theorem, hence 
 by Lemma \ref{lifting} there exists a continuous lift 
$$\Psi : I \to \widetilde{\mathbb{D}}_R \times \dots \times \widetilde{\mathbb{D}}_R.$$
Now, there exists an index $k$ between $1$ and $N$ such that $z$ 
is the $k^{th}$ coordinate of $\Psi(a)$; then the projection of 
$\Psi$ to the $k^{th}$ coordinate is the desired path $\gamma : I \to \widetilde{\mathbb{D}}_R$.
\end{proof}

Now, let us note that as an application of the previous Lemma, we can directly prove the 
\begin{proposition}
The set $\Sigma_{kn} \cap \mathbb{E}$ is path-connected. 
\end{proposition}

\begin{proof}
First, we know by Proposition \ref{annulus} that $\Sigma_{kn}$ 
contains an annulus of type $\{ 1 < |z| < R \}$ for some $R>1$: thus, if we pick 
$s \in \Sigma_{kn} \cap \mathbb{E}$ inside the annulus, then $s$ can 
be connected to the unit circle by a continuous path inside the annulus, which is contained in $\Sigma_{kn}$.
Otherwise, let us suppose $|s| > R$: then $z := s^{-1}$ belongs to the set 
$\widetilde{\Sigma} := \widehat{\Sigma}_{[-2, 1/4], R^{-1}}$.
By the previous Lemma, each element $(c, z) \in [-2, 1/4] \times \mathbb{D}_{R^{-1}}$
such that $K_c(z) = 0$ can be connected via a path inside $\widetilde{\Sigma}$ to an element $(0, w)$ of the 
fiber over $0$. Now, since $K_0(t)$ has no zeros inside the unit disk, then the fiber over $0$ contains only 
the point at infinity, thus $z$ is connected by a continuous path inside $\Sigma_{kn}^{-1}$ to a point on the boundary of the 
unit disk. 
\end{proof}

\subsection{Irreducibility and renormalization} \label{section:irredrenorm}

Now, in order to prove the path-connectivity of $\Sigma$ rather than of $\Sigma_{kn}$, 
we need to check what are the minimal polynomials of growth rates and whether they coincide
with the kneading polynomials.
It turns out that the set of all Galois conjugates is essentially the same as the 
set of zeros of kneading determinants of non-renormalizable parameters (see Proposition \ref{alguntuned}).

Recall that each superattracting map $f_c$ is the center of a 
small copy of the Mandelbrot set, which is the image of the Mandelbrot
set via a \emph{tuning} homeomorphism, as constructed by Douady and Hubbard. 
Let us denote by $\mathcal{M}_c$ the small Mandelbrot set with center $c$, and 
$U_c$ the interior of the real section of $\mathcal{M}_c$, i.e. the open real interval whose 
closure is $\mathcal{M}_c \cap \mathbb{R}$. Note that small Mandelbrot sets are either 
disjoint or nested, and in fact $c_1 \in \mathcal{M}_{c_2}$ implies 
$\mathcal{M}_{c_1} \subseteq \mathcal{M}_{c_2}$.

Recall moreover that $f_{-1}$ is the unique superattracting map
of period $2$, known as the \emph{basilica}. 
Let us denote by $\tau$ the operator given by tuning with the basilica, i.e. such that 
$$f_{\tau(c)} := f_{-1} \star f_c.$$
The operator $\tau$ will be also called \emph{period doubling} tuning operator
and will play a special role in the following; note the fixed point of $\tau$ is the Feigenbaum parameter
$c_{Feig}$.
Let us moreover define the set $N$ of \emph{non-renormalizable} parameters as 
$$N := [-2, 1/4] \setminus \bigcup_{c \in M_0 \cap \mathbb{R}} U_c,$$
that is the parameters which are not contained in the interior of any small Mandelbrot set; 
finally, let $N^\star$ be the set of successive period doublings of non-renormalizable 
parameters, i.e. 
$$N^\star := \bigcup_{n \geq 0} \tau^n(N).$$
As Douady pointed out, entropy behaves nicely with respect to renormalization; more precisely, 
as soon as the root of a small Mandelbrot set has positive entropy, then all maps in the same small Mandelbrot set 
have the same entropy:

\begin{proposition}[\cite{Do}] \label{sameentropy}
Let $c \in [-2, 1/4]$ belong to the small Mandelbrot set $\mathcal{M}_{c_0}$, and suppose that $s(f_{c_0}) > 1$. Then we have 
$$s(f_c) = s(f_{c_0}).$$
\end{proposition}

Note that the only real hyperbolic components with zero entropy are the ones which arise from 
the main cardioid after finitely many period doubling bifurcations, which explains why we need to 
consider the set $N^\star$. 
We claim that $\Sigma$ can be given the following characterization in terms of kneading determinants. 

\begin{proposition} \label{alguntuned}
The set $\Sigma \cap \mathbb{E}$ can be characterized in the following way: 
$$\Sigma \cap \mathbb{E} = \{ z \in \mathbb{E} \ : \ K_c(1/z) = 0 \ \textup{for some } c \in N^\star\}.$$
\end{proposition}

The proof will use the following lemma, whose proof we postpone to section \ref{section:density}.

\begin{lemma} \label{densityprimes}
Let $c \in N$ be a non-renormalizable parameter, with $c < -1$. 
Then there exists a sequence $c_n \to c$ of superattracting parameters whose kneading polynomials 
$P_{c_n}(t)$ are irreducible.
\end{lemma}

\begin{proof}[Proof of Proposition \ref{alguntuned}.]
Let $z \in Gal(s(f_c))$, with $|z| > 1$ and $f_c$ superattracting. Then 
we can write $c = \tau^n(c_\star)$ with $n \geq 0$ and $c_\star \in [-2, \tau(-2))$. 
Moreover, let $c_0$ be the root of the maximal small Mandelbrot set which contains
$c_\star$; note that $c_0$ belongs to $N$, and $c_1 := \tau^n(c_0)$ belongs to $N^\star$. 
Moreover, note that $c_1 < c_{Feig}$ so $s(f_{c_1}) > 1$ and 
by Proposition \ref{sameentropy} we have that 
$$s(f_c) = s(f_{c_1}).$$ 
Thus, the growth rate $s(f_c)$ is a root of the reciprocal of the kneading polynomial $P_{c_1}(t)$, 
which has integer coordinates, so its Galois conjugate $z$ is also root of the same polynomial
and $K_{c_1}(1/z) = 0$.

Conversely, let $z \in \mathbb{E}$ such that $K_c(1/z) = 0$ for some $c \in N^\star$. 
Suppose first that $c$ in non-renormalizable: then, by Lemma \ref{densityprimes} there exists a sequence 
$c_n \to c$ of superattracting parameters such that the period $p_n$ of $f_{c_n}$ is a power of $2$ 
and the polynomials $P_{c_n}(t)$ are irreducible. 
Since $K_{c_n}(z) \to K_c(z)$ inside the unit disk, by Rouch\'e's theorem 
there exists a sequence $z_n \to z$ such that for each $n$ we have $K_{c_n}(z_n^{-1}) = 0$, 
so each $z_n$ is a root of the reciprocal of the polynomial $P_{c_n}(t)$,
 which is also irreducible, 
so $z_n$ belongs to $Gal(s(f_{c_n})) \subseteq \Sigma$, and the claim follows by taking the limit. 

If instead $c$ is of the form $c = \tau^k(c_\star)$, with $c_\star \in N$, then by applying the 
same reasoning for $c_\star$ we can find a sequence $c_n \to c_\star$ of superattracting maps
with irreducible kneading polynomials $P_{c_n}(t)$; then we have also $\widetilde{c}_n := \tau^k(c_n) \to c$ and that
there exists a sequence $z_n \to z$ such that $K_{\widetilde{c}_n}(z_n^{-1}) = 0$. 
Now, for each $n$ we have that 
$$P_{\widetilde{c}_n}(t) = (1-t)(1-t^2)\dots(1-t^{2^{k-1}})P_{c_n}(t^{2^k})$$
so the growth rate $s_n := s(f_{\widetilde{c}_n})$ is a zero of the reciprocal of the kneading 
polynomial $P_{c_n}(t^{2^k})$, which is irreducible by Lemma \ref{irredhigh}, so $z_n$ 
belongs to $Gal(s_n) \subseteq \Sigma$, and the claim follows. 
\end{proof}

Our goal to show path-connectivity is to apply the path-lifting lemma by using $N^\star$
as the base space; however, $N^\star$ is totally disconnected, so we cannot apply the argument directly 
as in section \ref{section:lifting}. 

We shall now see that we can reduce ourselves to taking as our base space a set of parameters with 
finitely many connected components, and then we shall apply the argument to each component. 
For an integer $p$, define $N_p$ to be the complement of the interiors of the small 
Mandelbrot sets of period less than $p$:
$$N_p := [-2, 1/4] \setminus \bigcup_{\stackrel{c\in M_0 \cap \mathbb{R}}{\textup{Per}(f_c) < p}} U_c.$$
Since there are only finitely many hyperbolic components of a 
given period, the set $N_p$ is a finite union of closed intervals. 
Moreover, given $n > 0$ let us define $N_{p, n}$ to be the union
$$N_{p, n} := \bigcup_{k = 0}^n \tau^k(N_p).$$
The set $N_{p, n}$ is a finite union of closed intervals.
Given $I \subseteq [-2, 1/4]$ a closed interval, we define the set
$$\Sigma_{I, R} :=  \{ z \in \mathbb{D}_R \ : \ K_c(z) = 0 \textup{ for some }c \in I\}.$$

\begin{lemma} \label{finitecomp}
There exist real constants $0 < R_0 < R < 1$ and positive integers $p, n$ such that 
we have the equality
$$\Sigma^{-1} \cap \mathbb{D} = \{ z \ : \  R_0 < |z| < 1 \} \cup \bigcup_{k = 1}^r \Sigma_{I_k, R}$$
where $I_1, \dots, I_r$ are the connected components of $N_{p, n}$.
\end{lemma}

\begin{proof}
By Proposition \ref{annulus}, the set $\Sigma$ contains an annulus, so we can choose $R_0 < 1$ such that 
$\Sigma^{-1}$ contains the set $\{ z \ : \ R_0 < |z| < 1 \}$, and let us choose any 
$R \in (R_0, 1)$. Moreover, there exists a positive integer $p$ such that $2^{-1/p} > R$, and similarly 
there exists $n$ such that $2^{-1/2^n} > R$. We shall see that the claim holds with these choices.

Let now $z \in \Sigma^{-1} \cap \mathbb{D}$. By Proposition \ref{alguntuned} we have that 
there exists $c \in N^\star$ such that $K_c(z) = 0$. Now either $|z| > R_0$, or $|z| \leq R_0 < R$; 
in the latter case, we have $c = \tau^k(c_\star)$ with $k \geq 0$ and $c_\star \in N \subseteq N_p$.
Thus, $z$ is a root of $K_{c_\star}(t^{2^k})$ hence by Lemma \ref{basicincl} we have  $|z^{2^k}| \geq 1/2$
so by our choice of $n$ we must have $k \leq n$ and $c \in N_{p, n}$.

Conversely, if $z$ belongs to $\Sigma_{I_k, R}$ then there exists $c \in N_{p, n}$ such that $K_c(z) = 0$.
Thus we can write $c = \tau^k(c_\star)$, with $k \leq n$, so that $c_\star$ does not lie in the interior of a 
small Mandelbrot set of period less than $p$. If $c_\star$ is non-renormalizable, then $c$ belongs to $\tau^k(N) \subseteq N^\star$, 
hence $z^{-1}$ belongs to $\Sigma \cap \mathbb{E}$ by Proposition \ref{alguntuned}.
Otherwise, let $c_0$ be the root of the maximal small Mandelbrot set containing $c_\star$, 
and let $c_1 := \tau^k(c_0)$. Note that by construction 
the period of $c_0$ is at least $p$, so also the period $p_0$ of $c_1$ is at least $p$; 
moreover, $c_1$ belongs to $N^\star$.
Now we can write $f_{c} = f_{c_1} \star f_{c_2}$ for some $c_2$, hence the kneading determinant is 
$$K_{c}(t) = P_{c_1}(t) K_{c_2}(t^{p_0}).$$
Now, by Lemma \ref{basicincl} all zeros of $K_{c_2}(t^{p_0})$ lie outside the circle of radius 
$2^{-1/p_0} \geq 2^{-1/p} > R$, so $z$
must be a zero of $P_{c_1}(t)$ and hence of $K_{c_1}(t)$ with $c_1 \in N^\star$, thus it belongs to 
$\Sigma^{-1}$ by Proposition \ref{alguntuned}.
\end{proof}

In order to prove the path-connectivity of $\Sigma \cap \mathbb{E}$ we will need to apply 
the path-lifting lemma to each $\Sigma_{I_k, R}$. Let us see the proof in detail.



\begin{proof}[Proof of Theorem \ref{sigmapc}]
Recall that $\Sigma$ contains the annulus $A_{R_0} := \{ 1 < |z| < R_0 \}$ for some $R_0 >1$. 
We shall show that every point of $\Sigma \cap \mathbb{E}$ can be connected 
to the annulus via a continuous path contained in $\Sigma$. 
Let $z \in \Sigma$, which we can assume such that $|z| \geq R_0$. By Lemma \ref{finitecomp}, 
there exist $R < 1$ and integers $n, p$ such that $z^{-1}$ belongs to 
$$\bigcup_{k = 1}^r \Sigma_{I_k, R}$$
where $I_1, \dots, I_r$ are the connected components of $N_{p, n}$ (labelled so that 
$I_1 < I_2 < \dots < I_r$ in the ordering of the real line). We shall
denote by $[\alpha_k, \beta_k]$ the endpoints of $I_k$, and for each parameter $c$ 
we will denote as
$$F_c := \{ z \in \mathbb{D}_R \ : \ K_c(z) = 0 \} \cup \{ \infty \}$$ 
the fiber over $c$. 
Thus $z^{-1}$ belongs to some $\Sigma_{I_k, R}$, and there exists $c \in I_k$ such that 
$(c, z^{-1})$ belongs to $\widehat{\Sigma}_{I_k, R}$.
Using Lemma \ref{liftknead}, let us lift the interval $[\alpha_k, c]$ to a continuous 
path in $\widehat{\Sigma}_{I_k, R}$ joining $(c, z^{-1})$ to a point $(\alpha_k, z_1)$ on the fiber
over $\alpha_k$. If $k = 1$ we stop, otherwise we wish to ``continue'' the path 
to the interval $I_{k-1}$ to the left.
In order to do so, note that if $U_c := (c_1, c_0)$ is the real section of a 
small Mandelbrot set of period $p$, 
then the kneading determinants have the following form:
$$K_{c_0}(t) = \frac{P_{c_0}(t)}{1-t^p} \qquad K_{c_1}(t) = \frac{P_{c_0}(t)(1-2t^p)}{1-t^p}$$
so we have the inclusion between the fibers $F_{c_0} \subseteq F_{c_1}$. This is the key 
step to continue the path to the neighbouring component. 
Indeed, since the fiber $F_{\alpha_k}$ is a subset of the fiber $F_{\beta_{k-1}}$, 
we can lift the interval $[\alpha_{k-1}, \beta_{k-1}]$
to a continuous path in $\widehat{\Sigma}_{I_{k-1}, R}$ joining $(\beta_{k-1}, z_1)$ to some point $(\alpha_{k-1}, z_2)$
on the fiber over $\alpha_{k-1}$. 
By iterating this procedure, we find a sequence of $k$ continuous paths $\gamma_a : [0,1] \to \widehat{\Sigma}_{I_{k-a}, R}$
with $a = 0, \dots, k-1$, and points $z_1, \dots, z_k$ such that 
$$\begin{array}{ll}
\gamma_0(0) = (c, z^{-1}) & \\
\gamma_a(0) = (\beta_{k-a}, z_{a}) & \textup{ for }1 \leq a \leq k-1 \\ 
\gamma_a(1) = (\alpha_{k-a}, z_{a+1}) & \textup{ for }0 \leq a \leq k-1.  
\end{array}$$
Now, if we denote $\pi_2 : [-2, 1/4] \times \widetilde{\mathbb{D}}_R \to \widetilde{\mathbb{D}}_R$ the projection onto the 
second coordinate, 
then the projected path 
$\gamma := \pi_2(\gamma_0 \cup\gamma_2\cup \dots \cup \gamma_{k-1})$ is a continuous path inside $\Sigma^{-1} \cap \widetilde{\mathbb{D}}_R$
starting from $z^{-1}$: if $\gamma$ hits the boundary of $\mathbb{D}_R$, then by taking inverses we get that $z$ is connected 
via a path inside $\Sigma$ to 
the annulus $\{ z : 1< |z| < R_0 \}$. Otherwise, $z^{-1}$ is connected to the projection of the endpoint $\gamma_{k-1}(1)$, 
which by definition belongs to the fiber $F_{-2}$. However, we know by computation that  
$$K_{-2}(t) = \frac{1-2t}{1-t},$$
so the fiber $F_{-2}$ is the union of the boundary of $\mathbb{D}_R$ with the singleton $1/2$.
If $\pi_2(\gamma_{k-1}(1))$ does not lie on the boundary of $\mathbb{D}_R$, then $z^{-1}$ is connected via a continuous 
path inside $\Sigma^{-1}$
to $1/2$, hence after inversion $z$ is connected to $2$, and we know (Corollary \ref{contains12}) that 
set $\Sigma$ contains the real interval $[1,2]$, so $z$ can also be connected to the annulus $A_{R_0}$
by a continuous path inside $\Sigma$. 
\end{proof}

The last step to complete the proof of Theorem \ref{main} is the following.

\begin{proposition}
The set $\Sigma \cap \mathbb{E}$ is locally connected.
\end{proposition}

\begin{proof}
Let $I_1, \dots I_r$ be the connected components of $N_{p, n}$ as in Lemma \ref{finitecomp}.
For each $k$, by applying Lemma \ref{lclemma} with $\Lambda = I_k$, $V = \widetilde{\mathbb{D}}_R$ and $U = \mathbb{D}_R$
we get that $\Sigma_{I_k, R}$ is locally connected. As a consequence, since every set $\Sigma_{I_k, R}$ is 
closed in $\mathbb{D}_R$, then the finite union $\bigcup_{k = 1}^r \Sigma_{I_k, R}$ is also locally connected.
Thus, the union $\{ z \ : \  R_0 < |z| < 1 \} \cup\bigcup_{k = 1}^r \Sigma_{I_k, R}$ is locally connected, and its inverse
coincides with $\Sigma \cap \mathbb{E}$ by Lemma \ref{finitecomp}.
\end{proof}

\subsection{A remark on irreducibility} \label{s:remirr}

Note that the irreducibility of the kneading polynomials is a very delicate issue. 
In fact, if the parameter $c$ is renormalizable, (e.g. if the dynamical system $f_c$ ``splits" into 
two dynamical systems) then $P_c(t)$ is reducible by eq. \eqref{eq:renorm}. On the other hand, 
there are also non-renormalizable maps for which the corresponding kneading polynomial
is reducible. 
For instance, the polynomial
\begin{equation}
P_c(t) := 1 - t- t^2 + t^3 - t^4 + t^5 - t^6
\end{equation}
is admissible, and reducible over $\mathbb{Z}[t]$ (in fact $P_c(t) = (1- t - t^3)(1-t^2 + t^3)$ ) 
but the corresponding map $f_c$ is not renormalizable (since it has period $7$).
In such cases, one can formulate the  

\textbf{Question.}
Does the above factorization of $P_c(t)$ arise from some form of splitting of the dynamics of the corresponding map $f_c$ ?

\medskip

Let us note moreover that if $p$ is an odd prime, then the kneading polynomials for real superattracting maps of period $p$
all reduce to the same cyclotomic polynomial $P_c(t) = 1 + t + t^2 + \dots + t^{p-1}$ modulo $2$, and such 
polynomial is irreducible over $\mathbb{Z}/{2 \mathbb{Z}}$ if and only if $2$ is a primitive root of unity modulo $p$ (i.e., $2^k \equiv 1 \mod p$ for all $k = \{1, \dots, p-2\}$).

In a similar spirit, one can study the number of irreducible polynomials with coefficients in the set $\{ \pm 1\}$. This question 
appears to be pretty hard, and is related to several conjectures in number theory (see also \cite{OP}).
More precisely, one can write 
$$\mathcal{I}(n) := \frac{\#\{ \epsilon \in \{\pm 1\}^n \ : \ P_\epsilon(t) := \sum_{i = 1}^{n}   \epsilon_i t^{i-1} \textup{ is irreducible } \}}{2^n}$$
and look at the asymptotic behavior of $\mathcal{I}(n)$. 
By Lemma \ref{irred}, we have $\limsup \mathcal{I}(n) \geq 1/2$; note that for instance,
Artin's primitive root conjecture implies that $\limsup \mathcal{I}(n) = 1$, but in general the question appears to be open; 
in the case of coefficients $0, 1$, the fact that almost all such polynomials are irreducible is due to Konyagin \cite{Ko}.

\section{Dominant strings and hyperbolic components with irreducible kneading polynomial} \label{section:density}

We shall now present the proof of Lemma \ref{densityprimes}.
In order to do so, let us recall some notation on the combinatorics of kneading sequences
we introduced in \cite{Ti}. Let $S = (a_1, \dots, a_n)$ be a finite sequence of positive integers, 
which we will sometimes call a \emph{string}. 
For reasons which will become clear in a moment, the \emph{period} of $S = (a_1, \dots, a_n)$ 
will be the sum of the digits $p(S) := a_1 + \dots + a_n$.
We endow the set of strings with the following partial order.
If $S = (a_1, \dots, a_n)$ and $T = (b_1, \dots, b_m)$ are two finite strings of positive integers, we 
write $S << T$ if there exists a positive index $k \leq \min\{m, n\}$ such that 
$$a_i = b_i \textup{ for all }1 \leq i \leq k-1, 
\textup{  and }\left\{ \begin{array}{ll} a_i < b_i & \textup{if }k\textup{ even} \\
                      a_i > b_i & \textup{if }k \textup{ odd.}
                     \end{array}
                     \right.$$  
A string $S$ of even length is called \emph{dominant} if it is smaller than all its suffixes: 
namely, $S$ is dominant if for each non-trivial splitting $S = XY$ in two substrings $X$ and $Y$, 
one has
$$S << Y.$$              
One should think of this order as an alternate lexicographical order; for instance, $(2, 1) << (1)$
but $(2, 1) << (2, 2)$, while the strings $(2)$ and $(2, 3)$ are not comparable.

The following facts about dominant strings are easily checked:
\begin{enumerate}
 \item if $S = (a_1, \dots, a_m)$ is dominant and $a_1 > 1$, then for each $n \geq 1$ the string 
 $S^n11$ is dominant; 
 \item if $S$ and $T$ are dominant strings and $S << T$, then for each $n \geq 1$ the string $S^n T$ is 
 dominant.
\end{enumerate}

The reason we define dominant strings is that they allow us to construct admissible 
kneading sequences; namely, if $S = (a_1, \dots, a_n)$ is a dominant string, then 
by the criterion of Theorem \ref{admthm} there exists a superattracting real parameter $c$ of period $p(S) = a_1 + \dots + a_n$
with kneading polynomial 
$$P_c(t) = 1 + \sum_{k = 1}^n (-1)^k \sum_{j = a_1 + \dots + a_{k-1} + 1}^{a_1 + \dots + a_k} t^j - t^{a_1 + \dots + a_n}.$$
Such a superattracting parameter will be called a \emph{dominant parameter}. For instance, 
the ``airplane map'' of period $3$ has kneading polynomial $P_c(t) = 1 - t - t^2$, and its corresponding string 
is $S = (2, 1)$ which is dominant. Thus, the airplane parameter is dominant.
Furthermore, we shall call \emph{index} of the string $S = (a_1, \dots, a_n)$ the alternating sum 
$\llbracket S \rrbracket := \sum_{k = 1}^n (-1)^k a_k$.

\begin{proof}[Proof of Lemma \ref{densityprimes}]
By (\cite{Ti}, Lemma 11.5) every non-renormalizable parameter $c < -1$ can be approximated by a dominant parameter.
For this reason, it is enough to prove that dominant parameters can be approximated by centers of 
hyperbolic components with irreducible kneading polynomial.

Note now that, if $S$ is the string associated to a dominant parameter, 
in order to prove that the corresponding kneading polynomial is irreducible
it is sufficient, by Lemma \ref{irred}, to check the two following conditions:

\begin{enumerate}
 \item[(i)] the period $p(S) = 2^N$ for some $N$; 
 \item[(ii)] the index $\llbracket S \rrbracket \equiv 2 \mod 4$.
\end{enumerate}

Let now $c$ be a dominant parameter with associated dominant string $S$, and define 
for any pair of positive integers $a$, $b$ the string 
$$Z_{a, b} := 2\underbrace{1\dots1}_{a \textup{ times}}2\underbrace{1\dots1}_{b \textup{ times}}.$$
It is immediate to check that, if $a$ and $b$ are odd and $a < b$, then the string 
$Z_{a, b}$ is dominant.
We shall see that $c$ can be approximated by a sequence of superattracting parameters 
whose associated strings are of the form $S^n Z_{a, b}$ and satisfy (i) and (ii), 
hence their kneading polynomials are irreducible.

Indeed, since $c$ is non-renormalizable, then it must lie 
 outside the small Mandelbrot set determined by the basilica component, hence 
$c < \tau(-2)$, which in the language of strings translates into the inequality
$\overline{S} < (2, \overline{1})$. As a consequence, for any sufficiently large 
 odd integer $a$ we have $S << 2\underbrace{1\dots1}_{a \textup{ times}}$, 
 hence also 
 $$S << Z_{a, b}$$
 and for each $n$ the string $S^n Z_{a, b}$ is dominant.
 On the other hand, if $n$ is multiple of $4$, then the index of $S^n Z_{a, b}$ is 
 $$\llbracket S^n Z_{a, b} \rrbracket = n \llbracket S \rrbracket + 2 \equiv 2 \mod 4$$
 which satisfies (ii).
 Finally, the period of $S^n Z_{a, b}$ is 
 $$p(S^n Z_{a, b}) = n p(S) + a + b + 2$$
 hence for each $n$ one can choose $b = b_n > a$ such that the period is a power of $2$.
 Then all elements of the sequence $S^n Z_{a, b_n}$ with $n \equiv 0 \mod 4$ satisfy the conditions (i), (ii)
 hence their corresponding parameters converge to $c$ and their kneading polynomials are 
 irreducible.

\end{proof}

\section{Appendix}

We conclude with the proof of a few lemmas about convex sets, which are used in section \ref{section:neigh}.

\begin{proof}[Proof of Lemma \ref{paracrit}]
Let us first show that there exists a constant $c > 0$ such that 
\begin{equation}
\label{dotprod}
\max_{1 \leq i \leq n} \langle w, v_i \rangle \geq c \Vert w \Vert \qquad \forall w \in \mathbb{R}^d.
\end{equation}
Indeed, let $w$ be a vector of unit norm. Since the vector $0$ lies in the interior of the convex hull generated by the $v_i$, 
we can write $$0 = \sum_{i = 1}^n \lambda_i v_i \qquad \textup{for some }0 < \lambda_i < 1$$
hence by taking the dot product with $w$ we realize that there must exist an index $i$ such that $\langle v_i, w \rangle > 0$
(note that there exists an index $i$ such that $\langle w, v_i \rangle \neq 0$ since the $v_i$ span $\mathbb{R}^d$). Thus, equation 
\eqref{dotprod} holds by compactness of the unit ball and scaling. 
Let us now pick a constant $R > 0$, and let $x$ belong to the closure of the ball of radius $R$. If 
there exists an index $1 \leq i \leq n$ such that 
$\Vert x - v_i \Vert < \Vert x \Vert$, then $\Vert x - v_i \Vert < \Vert x \Vert \leq R$ and we are done. Otherwise, by 
writing the condition $\Vert x - v_i \Vert \geq \Vert x \Vert$ in 
terms of dot products we have for each $i$ the inequality $\langle x, v_i \rangle \leq \frac{\Vert v_i \Vert^2}{2}$ thus 
by combining it with equation \eqref{dotprod} we have 
$$c \Vert x \Vert \leq \max_{1 \leq i \leq n} \langle x,  v_i \rangle \leq \max_{1 \leq i \leq n} \frac{\Vert v_i \Vert^2}{2}$$ 
thus $x$ is bounded independently of $R$. As a consequence, it is enough to choose $R$ large enough so that the ball of
center $v_1$ and radius $R$ contains the ball centered at the origin with radius $\max_{1 \leq i \leq n} \frac{\Vert v_i \Vert^2}{2c}$.
\end{proof}

\begin{proof}[Proof of Lemma \ref{onecontained}.]
Let us first prove that $\Sigma$ contains a neighbourhood of $1$. 
Let $z$ be near $1$, and let $\delta := z - 1$. 
For $R > 0$, denote as $B_{\delta, R}$ the parallelogram of vertices 
$\{ \pm R \pm R \delta \}$; we claim that there exist $R > 0$, an integer $n > 0$ and 
a neighbourhood $U$ of $1$ in the complex plane such that for each $z$ in $U$ with $|z| < 1$ and 
non-zero imaginary part (so that the parallelogram is non-degenerate) we have the inclusion
\begin{equation} \label{incl_one}
B_{\delta, R} \subseteq  \bigcup_{\epsilon \in \{ \pm 1 \}^n} \left( \sum_{j = 0}^{n-1} 
\epsilon_j (-1 + z + z^2) z^{3j}  + z^{3n} B_{\delta, R} \right) 
\end{equation}
and moreover the point $u_6(z) := \frac{1 - z- z^2 - z^3- z^4-z^5}{z^6}$ belongs to $B_{\delta, R}$, 
from which the claim follows by Lemma \ref{boundarycrit}. 
The fundamental idea to prove equation \eqref{incl_one} is to perform the computation in a 
basis which changes as $z$ changes (as in \cite{B2}, Proposition 3.3). Namely, for each non-real $z$ in a neighbourhood of $1$, the 
set $\{1, \delta \}$ is an $\mathbb{R}$-basis for $\mathbb{C}$, and multiplication by $z$ is an $\mathbb{R}$-linear 
map which is represented in such a basis by the matrix 
$$T := \matr{1}{0}{1}{1}$$
up to $O(|\delta|)$ as $\delta \to 0$. Then, the point $(-1+z+z^2)z^{3j}$ is represented up to $O(|\delta|)$ by the vector 
$$V_j := -T^{3j} \vect{1}{0} + T^{3j+1} \vect{1}{0} + T^{3j+2} \vect{1}{0} = \vect{1}{3j+3}.$$
Finally, the set $B_{\delta, R}$ has in this basis the vertices $(\pm R, \pm R)$. We can now 
choose $R$ divisible by $4$ and large enough so that Lemma \ref{contained} holds for $m = 3$; then, 
we can choose $n$ so that Lemma \ref{lattice} holds, hence for $|\delta|$ small enough we have that 
equation \eqref{incl_one} holds, and the claim is proven.
 
Let us now pick $z$ close to $-1$; if we let $\delta := z + 1$, then multiplication by $z$ is given in the basis 
$\{1, \delta\}$ by the matrix $\widetilde{T}:= \matr{-1}{0}{1}{-1}$,  up to $O(|\delta|)$. The same argument 
works as for $z$ close to $1$: indeed, in this case we consider the parallelogram $\widetilde{B}_{\delta, R}$
of vertices $\{±Rz, ±R(2 + z)\}$, and by Lemma \ref{contained} we can choose $R$ large enough so that $u_{6}(z)$ 
is contained in $\widetilde{B}_{\delta, R}$. Moreover, let us note that if we choose 
$$\widetilde{V}_j := \widetilde{T}^{3j} \vect{1}{0} + \widetilde{T}^{3j+1} \vect{1}{0} - \widetilde{T}^{3j+2} \vect{1}{0} = 
(-1)^{3j} \vect{-1}{3j+3}$$
then we have 
$\widetilde{\Lambda}_n := \left\{ \sum_{j = 0}^{n-1} \epsilon_j \widetilde{V}_{j} \ : \ \epsilon_j \in \{ \pm 1 \} \right\} = \sigma(\Lambda_n)$ 
where $\sigma(x,y) := (x, -y)$ is the reflection through the $x$-axis. Moreover, if $B$ is the square of vertices of coordinates
$(\pm R, \pm R)$ one has $\sigma(B) = B$, and $\sigma(\frac{1}{2} T^{3n} B) = \frac{1}{2} \widetilde{T}^{3n} B$.
Thus, let us choose $n$ which satisfies Lemma \ref{lattice}, and by applying the reflection $\sigma$ we have 
$$B \subseteq \bigcup_{v \in \widetilde{\Lambda}_n}  v + \frac{1}{2} \widetilde{T}^{3n} B;$$
thus, if we interpret the inclusion in the basis $\{1, \delta\}$ we get for small $|\delta|$  
$$\widetilde{B}_{\delta, R} \subseteq  \bigcup_{\epsilon \in \{ \pm 1 \}^n} \left( \sum_{j = 0}^{n-1} 
\epsilon_j (1 + z - z^2) z^{3j}  + z^{3n} \widetilde{B}_{\delta, R} \right)$$ 
which proves the claim.

\end{proof}


In the following lemma, we will denote by $V_j$ the vector 
$V_j := \vect{1}{3j+3}$ and by $\Lambda_n$ the finite set 
$$\Lambda_n := \left\{ \sum_{j = 0}^{n-1} \epsilon_j V_{j} \ : \ \epsilon_j \in \{ \pm 1 \} \right\}.$$ 
\begin{lemma} \label{lattice}
Fix $R \geq 6$ with $R \equiv 0 \mod 4$, and let $B$ be the square of vertices $( \pm R, \pm R)$.
Then there exists a positive integer $n$ such that we have the inclusion
$$B \subseteq \bigcup_{v \in \Lambda_n}  \left( v + \frac{1}{2} T^{3n} B \right).$$
\end{lemma}

\begin{proof}
Let $n \geq 1$, and $a$ such that $|a| \leq n$, $a \equiv n \mod 2$.
An elementary computation shows that the set $\Lambda_n$ contains all elements 
of the form 

$$\left\{ (a, b)\in \mathbb{Z}^2 \ : \ m_{a,n} \leq b \leq M_{a,n}, \ b \equiv M_{a,n} \mod 6 \right\}$$ 
where 
$$
\begin{array}{l}
M_{a,n} := (6a(n+1)+ 3(n^2-a^2))/4 \\
m_{a,n} := (6a(n+1)- 3(n^2-a^2))/4. 
\end{array}$$
Now, since $R \equiv 0 \mod 4$, if we choose $n \equiv 0 \mod 4$ we have that $M_{R, n} \equiv 0 \mod 6$; 
if we choose $n$ large enough, then $(3n+3)R/2 \leq |m_{R, n}|$, so $\Lambda_n$ contains the set 
\begin{equation} \label{contains}
\left\{ (a, b)\in \mathbb{Z}^2  \ : \ a \in \{-R, 0, R\}, \ |b| \leq \frac{(3n+3)R}{2}, \ b \equiv 0 \mod 6 \right\}.
\end{equation}

Let now $(x, y) \in B$. Then there exists $a \in \{-R, 0, R\}$ such that $|x + a| \leq R/2$, 
and since $R \geq 6$ there exists $b$ multiple of $6$ such that $3n(x+a) - R/2 \leq y + b \leq 3n(x+a) + R/2$.
Thus, by construction the vector $(x,y) + (a, b)$ belongs to the parallelogram $\frac{1}{2} T^{3n} B$, 
which has vertices
$$\pm \frac{R}{2}\vect{1}{3n} \pm \frac{R}{2} \vect{0}{1}.$$
Moreover, since $y$ belongs to $B$, we have the inequality 
$$|b| \leq |y+b| + |y| \leq \frac{(3n+3)R}{2}$$
so by the previous discussion (see equation \eqref{contains}) the vector $(a, b)$ belongs to $\Lambda_n$ 
and the claim is proven.
\end{proof}

\begin{lemma} \label{contained}
Fix $m \geq 2$. Then for large enough $R > 0$ there exists a neighbourhood $U$ of $1$ such that 
for all $z \in U$ the point 
$$u_m(z) := \frac{1 - \sum_{k = 1}^{m-1} z^k}{z^m}$$ 
is contained in the parallelogram $B_{\delta, R}$ of vertices $\{\pm R z, \pm R(2-z) \}$. 
Similarly, for large enough $R > 0$ there exists a neighbourhood $\widetilde{U}$ of $-1$
such that for all $z \in \widetilde{U}$ the point $u_m(z)$ is contained in the parallelogram $\widetilde{B}_{\delta, R}$ 
of vertices $\{\pm R z, \pm R(2+z) \}$.
\end{lemma}

\begin{proof}
By an elementary calculation we have $u_m(z) = (2-m) + \frac{m^2-3m}{2}(z-1) + O(|z-1|^2)$
so the claim holds as long as $R > \max\{|2-m|, \frac{|m^2-3m|}{2}\}$. The second case is completely analogous.
\end{proof}

\medskip

\subsection*{Entropies of maps along veins of the Mandelbrot set.}

A set similar to $\Sigma$ can be constructed for any vein $v$ in the Mandelbrot set,
not necessarily real. Namely, for each superattracting parameter $c$ in the Mandelbrot set, one can 
consider the restriction of the map $f_c$ to its Hubbard tree, and its growth rate will be an algebraic number.
Thus, given any vein $v$ in the Mandelbrot set, one can plot the union of all Galois conjugates of all 
superattracting parameters which belong to $v$. Here we show the pictures for the principal veins in the 
$1/3$, $1/5$ and $1/11$-limbs (Figures \ref{galois3}, \ref{galois5} and \ref{galois11}). 

\begin{figure}[ht!]
 \centering
\includegraphics[width=0.85\textwidth]{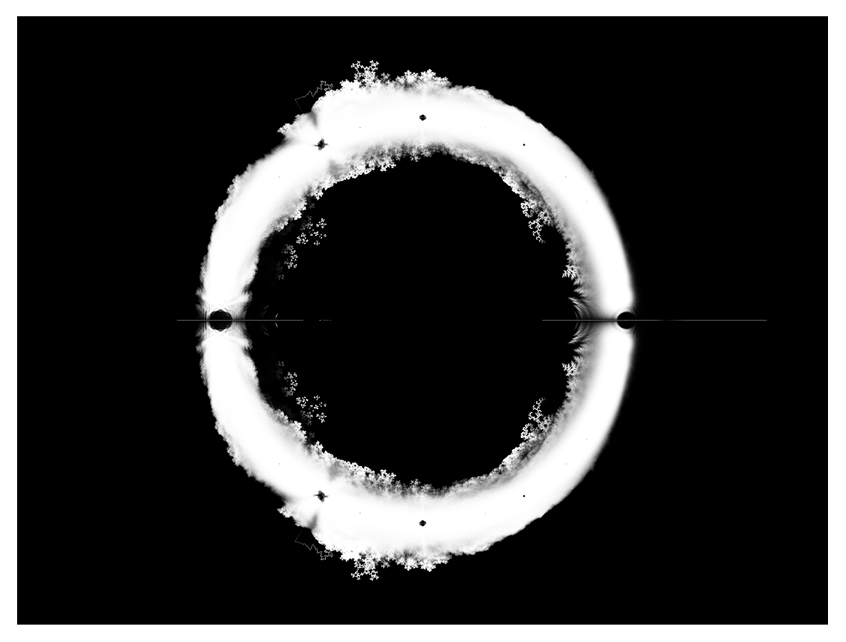}
\caption{Galois conjugates of entropies of maps along the vein $v_{1/3}$.}
\label{galois3}
\end{figure}

\begin{figure}[ht!]
 \centering
\includegraphics[width=0.85\textwidth]{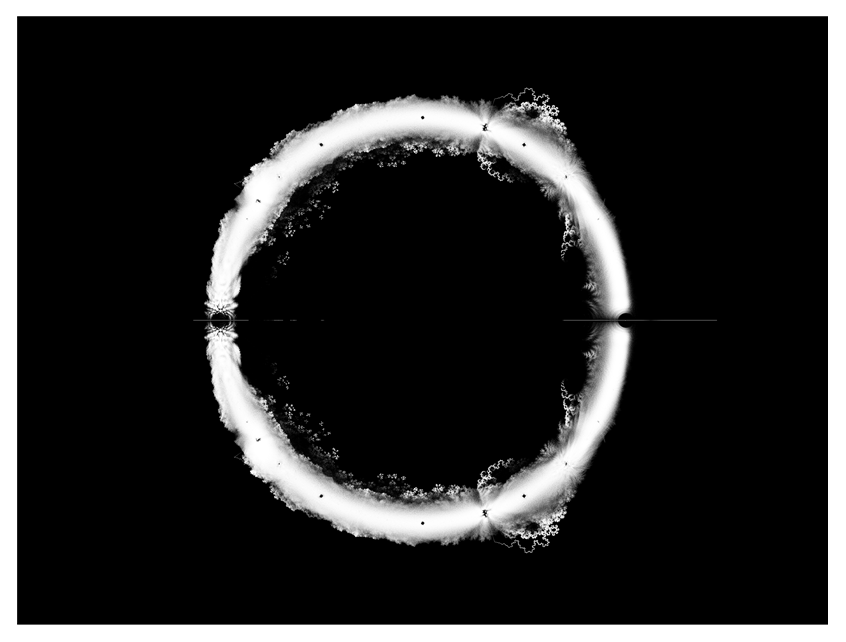}
\caption{Galois conjugates of entropies of maps along the vein $v_{1/5}$.}
\label{galois5}
\end{figure}

\begin{figure}[ht!]
 \centering
\includegraphics[width=0.85\textwidth]{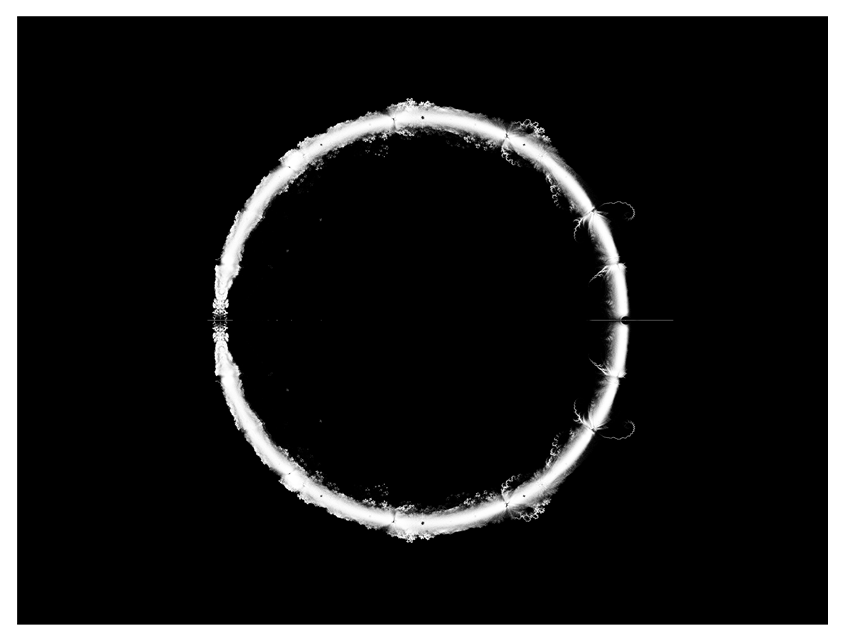}
\caption{Galois conjugates of entropies of maps along the vein $v_{1/11}$.}
\label{galois11}
\end{figure}

\end{document}